\documentclass{template_arxiv}

\usepackage{amsmath,amsfonts,amssymb,amsthm}
\usepackage{graphicx}
\usepackage{xcolor}
\usepackage{enumitem,hyperref,bm}

\newtheorem{theorem}{Theorem}%[section]
\newtheorem{lemma}{Lemma}%[section]
\newtheorem{proposition}[lemma]{Proposition}

\newtheorem{conjecture}[lemma]{Conjecture}

\theoremstyle{definition}

\newtheorem{remark}[lemma]{Remark}

\newtheorem{question}{Question}

\newcommand{\ds}{\displaystyle}

\newcommand{\N}{\mathbb{N}}

\newcommand{\R}{\mathbb{R}}
\newcommand{\A}{\mathcal{A}}
\newcommand{\B}{\mathcal{B}}

\newcommand{\aw}{\textcolor{white}{\mathtt{0}}}

\newcommand{\lan}{\mathcal{L}}

\newcommand{\abeq}{\sim_{\text{ab}}}

\newcommand{\colour}{\mathtt{color}}

\title{New examples of words for which binomial complexities and subword complexity coincide} 

\author[1]{L\'eo Vivion\thanks{lvivion.math@gmail.com}}

\affil[1]{Laboratoire de Math\'ematiques Pures et Appliqu\'ees Joseph Liouville, Universit\'e du Littoral C\^ote d'Opale, UR 2597, F-62100 Calais, France.}

\date{February 2026}

\newcommand{\shorttitle}{New examples of words for which binomial complexities and subword complexity coincide.}

\newcommand{\shortauthor}{L. Vivion}

\begin{document}

\maketitle

\begin{abstract}
The complexity of an infinite word can be measured in several ways, the two most common measures being the subword complexity and the abelian complexity. In 2015, Rigo and Salimov introduced a family of intermediate complexities indexed by $k\in\mathbb{N}_{>0}$: the $k$-binomial complexities. These complexities scale up from the abelian complexity, with which the $1$-binomial complexity coincides, to the subword complexity, to which they converge pointwise as $k$ tends to $+\infty$. In this article, we provide four classes of $d$-ary infinite words---namely, $d$-ary $1$-balanced words, words with subword complexity $n\in\mathbb{N}_{>0}\mapsto n+(d-1)$ (which form a subclass of the so-called quasi-Sturmian words), hypercubic billiard words, and words constructed by repeated Sturmian colorings---for which this scale ``collapses'', that is, all $k$-binomial complexities, for $k\geq 2$, coincide with the subword complexity. This work generalizes a result of Rigo and Salimov, established in their seminal paper from 2015, which asserts that the $k$-binomial complexity of any Sturmian word coincides with its subword complexity whenever $k\geq 2$.

\vspace*{.5cm}
\end{abstract}

%%%%%%%%%%%%%%%%%%%%%%%%%%%%%%%%%%%%%%%%%%
%
%     Fin des infos generales ICI 
%
%%%%%%%%%%%%%%%%%%%%%%%%%%%%%%%%%%%%%%%%%%

%%%%%%%%%%%%%%%%%%%%%%
\section{Introduction}

The subword complexity and the abelian complexity are two very classical ways to measure how ``complicated'' an infinite word is. In 2015, Rigo and Salimov introduced a family of complexities forming a scale between the abelian complexity and the subword complexity: the \emph{$k$-binomial complexities} \cite{RS15}. These complexities are parameterized by the integer $k\in\N_{>0}$ and are defined via a family of equivalence relations called \emph{$k$-binomial equivalences}. For $k=1$, this relation corresponds to the \emph{abelian equivalence}, and the $1$-binomial complexity is the same as the \emph{abelian complexity}. The $(k+1)$-binomial equivalence is a refinement of the $k$-binomial one, and as $k$ increases, the $k$-binomial equivalence gets progressively closer to equality between finite words. Consequently, the $k$-binomial complexities form a scale from the abelian complexity to the subword complexity, and the $k$-binomial complexity of a given word converges pointwise to its subword complexity as $k$ tends to $+\infty$.

\medskip

Since 2015, $k$-binomial complexities have been the subject of several papers \cite{LLR20,LRR20,LCWW24,RSW24,GRW24}. Interested readers may also consult \cite{RSW23,RRW25} for variations on $k$-binomial complexities and \cite{LRR20b,Whi21} for recent articles on $k$-binomial equivalence relations.

\medskip

Sturmian words constitute a central class in Combinatorics on Words. Their subword complexity is given by the function $n\in\N\mapsto n+1$ \cite{MH40}, while their abelian complexity is the constant function $n\in\N_{>0}\mapsto 2$ \cite{CH73}. In their seminal paper, Rigo and Salimov proved that Sturmian words satisfy the following remarkable combinatorial property: the $2$-binomial complexity of any Sturmian word is $n\in\N\mapsto n+1$, \emph{i.e.}, it coincides with their subword complexity. Note that, since the $k$-binomial complexities form a scale between the abelian complexity and the subword complexity, if the $k$-binomial complexity of a word coincides with its subword complexity, then this property also holds for all its $l$-binomial complexities with $l\geq k$.

\medskip

This article contributes to the study of words for which, similarly to Sturmian words, the $k$-binomial complexity coincides with the subword complexity for some small integer $k$. Although we focus mainly on the case $k=2$, several of our results also hold for arbitrary integer $k$ (see Lemmas~\ref{lemma:projections_kbin} and \ref{lemma:stability} below).

\medskip

Words whose $1$-binomial complexity (\emph{i.e.}, abelian complexity) coincides with their subword complexity can be easily characterized. Indeed, this condition is so restrictive that very few words satisfy it.

\begin{proposition}[\cite{RSW24} Remark~7.1]\label{prop:1bin}
The $1$-binomial complexity of an infinite word $w$ coincides with its subword complexity if and only if there exist $d$ distinct letters $a_1,\ldots,a_d$ and $(d-1)$ positive integers $k_1,\ldots,k_{d-1}$ such that
\[
    w=a_1^{k_1}a_2^{k_2}\ldots a_{d-1}^{k_{d-1}}a_d^\omega,
\]
where $a_d^\omega$ denotes the constant infinite word $a_da_da_da_d\ldots$.
\end{proposition}

By contrast, characterizing words whose $2$-binomial complexity coincides with their subword complexity is considerably more difficult and remains an open problem. In the literature, apart from the Sturmian case and the case of words whose $1$-binomial complexity coincides with their subword complexity, the only other known example is the Tribonacci word \cite{Rau82}, which is the most famous and widely studied instance of the class of \emph{Arnoux--Rauzy words} \cite{AR91}.

\begin{proposition}[\cite{LRR20} Theorem~29]
The $2$-binomial complexity of the Tribonacci word coincides with its subword complexity.
\end{proposition}

It is worth mentioning that the proof provided by Lejeune, Rigo and Rosenfeld is computer-assisted. At the time of writing, obtaining a non-computer-assisted proof of this result remains an open problem.

\medskip

Arnoux--Rauzy words (and, more generally, strict episturmian words \cite{DJP01,GJ09}) form a class of ternary (resp. $d$-ary) infinite words that can be viewed as a combinatorial and arithmetic generalization of Sturmian words. Lejeune, Rigo and Rosenfeld conjectured that all Arnoux--Rauzy words have their $2$-binomial complexity equal to their subword complexity (see \cite{LRR20} and \cite[Chapter~4, Section~4.4 and Appendix~B]{LejPhD}; see also Section~\ref{sec:conclusion} of this article where some numerical experiments are discussed).

\begin{conjecture}[Lejeune, Rigo, Rosenfeld]
The $2$-binomial complexity of any Arnoux--Rauzy word coincides with its subword complexity.
\end{conjecture}

Although the Tribonacci word is the only known example of an Arnoux--Rauzy word satisfying this property, the computer-assisted proof developed for it could potentially be applied to other purely morphic Arnoux--Rauzy words to determine whether this property holds in those cases as well (see \cite[Section~6]{LRR20}).

\medskip

Our main contributions are summarized in the following theorem, which provides several new and broad classes of words whose $2$-binomial complexity coincides with their subword complexity.

\begin{theorem}\label{th:main}
Let $d\geq 2$.\\
\emph{(i)} If $w$ is a $1$-balanced $d$-ary word, then its $2$-binomial complexity is equal to its subword complexity.\\
\emph{(ii)} If $w$ is a word with subword complexity $n\in\N_{>0}\mapsto n+(d-1)$, then its $2$-binomial complexity is equal to its subword complexity.\\
\emph{(iii)} If $w$ is a hypercubic billiard word in dimension $d$, then its $2$-binomial complexity is equal to its subword complexity.\\
\emph{(iv)} If $w$ is a $d$-ary Sturmian coloring, then its $2$-binomial complexity is equal to its subword complexity. 
\end{theorem}

These four classes of words are defined in Section~\ref{sec:def}. Interestingly, all these examples can be thought of as generalizations of Sturmian words. At the time of writing, apart from the trivial case described above where the $1$-binomial complexity coincides with the subword complexity, we are not aware of any word genuinely unrelated to Sturmian words and whose $2$-binomial complexity coincides with its subword complexity.

\medskip

The proof of Theorem~\ref{th:main} relies on the two key Lemmas~\ref{lemma:binary_projections} and \ref{lemma:stability} stated in Section~\ref{sec:main}. These lemmas provide sufficient conditions for the $2$-binomial complexity of a word to coincide with its subword complexity, and can be regarded as a first step toward a characterization of words satisfying this property. More precisely:
\begin{itemize}
    \item[-] Statements (i)--(iii) of Theorem~\ref{th:main} follow from Lemma~\ref{lemma:binary_projections}, which states that if all the binary projections of a given word are $1$-balanced, then its $2$-binomial complexity coincides with its subword complexity.
    \item[-] Statement (iv) of Theorem~\ref{th:main} follows from Lemma~\ref{lemma:stability}, which asserts that words whose $k$-binomial complexity equals their subword complexity are stable under a coloring process.
\end{itemize}

Finally, let us mention that for every $k\geq 3$, it is known that there exist words whose $k$-binomial complexity, but not their $(k-1)$-binomial complexity, coincides with their subword complexity. This was proved by Rigo, Stipulanti and Whiteland in \cite{RSW24}, where they showed that the image of any Sturmian word under $(k-2)$ iterations of the Thue--Morse substitution satisfies this property, see \cite[Theorems~7.2 and 7.4]{RSW24}.

\medskip

\noindent\textbf{Outline.} This paper is organized as follows. We recall the necessary definitions and the notation in Section~\ref{sec:def}. In Section~\ref{sec:main}, we present the general structure of the proof of Theorem~\ref{th:main}. In particular, we state the two key Lemmas~\ref{lemma:binary_projections} and \ref{lemma:stability} mentioned above, and discuss the optimality of each assertion. Then, in Section~\ref{sec:binary}, we prove that the $2$-binomial complexity of any binary $1$-balanced word is equal to its subword complexity. This covers statements (i)--(iii) of Theorem~\ref{th:main} in the case $d=2$. Section~\ref{sec:binary_projections} is devoted to the proof of Lemma~\ref{lemma:projections_kbin} (a generalization of Lemma~\ref{lemma:binary_projections}), as well as to statements (i)--(iii) of Theorem~\ref{th:main} (for $d\geq 3$). Section~\ref{sec:stability} is dedicated to the proof of Lemma~\ref{lemma:stability} and statement (iv) of Theorem~\ref{th:main}. Finally, we gather some open questions in Section~\ref{sec:conclusion}.

%%%%%%%%%%%%%%%%%%%%%%%%%%%%%%%%%%%%%%%%%%%%%%%%%%
\section{Definitions and notation}\label{sec:def}

% % % % % % % % % % % % % % % % % % %
\subsection{Basics and complexities}

\noindent\textbf{Basics.} Let $\A$ be a finite set, referred to as an \emph{alphabet}. A \emph{finite word} $w$ written over the alphabet $\A$ is an element of $\A^*:=\bigcup_{i=0}^{\infty}\A^i$, where $\A^0=\{\epsilon\}$ with $\epsilon$ denoting the empty word. For any $k\in\N_{>0}$, we write $\A^{\leq k}:=\bigcup_{i=0}^{k}\A^i$ for the set of all finite words of length at most $k$. A (right) \emph{infinite word} $w$ written over the alphabet $\A$ is an element of $\A^\N$.

A \emph{factor} or \emph{contiguous subword} $u$ of length $n$ of a word $w$ is a finite word formed of $n$ consecutive letters of $w$, while a \emph{scattered subword} of length $n$ of $w$ is a word consisting of $n$, not necessarily consecutive, letters of $w$. For instance, if $w=11212$, then $11$ is a factor (and therefore also a scattered subword) of $w$, $22$ is a scattered subword but not a factor of $w$, and $221$ is neither a factor nor a scattered subword of $w$. The \emph{language of $w$}, denoted $\lan(w)$, is the set of all factors of $w$. We also denote by $\lan_n(w):=\lan(w)\cap\A^{n}$ the set of length $n$ factors of $w$.

We write $w[n]$ for the letter occurring at position $n$ in the word $w$ (the indexation starts at $0$), and $w[n:m]$ for the length $m-n+1$ factor of $w$ starting at position $n$ and ending at position $m$, both included.

For a finite word $w\in\A^*$, we denote by $|w|$ its length, and by $|w|_a$ the number of occurrences of the letter $a\in\A$ in $w$. More generally, for a finite word $u\in\A^*$, we denote by $|w|_u$ the number of occurrences of $u$ in $w$ as a factor, and by $\binom{w}{u}$ the number of its occurrences in $w$ as a scattered subword. For instance, if $w=11212$, then $|w|=5$, $|w|_1=3$, $|w|_{12}=2$, and $\binom{w}{12}=5$. Note that the empty word has length zero: $|\epsilon|=0$.

\begin{remark}\label{rk:binomial_coefficients}
(i) Counting how many times a letter $a\in\A$ appears in $w$, either as a factor or as a scattered subword, is the same. Hence, $\binom{w}{a}=|w|_a=\binom{|w|_a}{1}$, where the last term refers to the usual binomial coefficient over the integers.\\
(ii) More generally, for a word of the form $a^m$ with $m\in\N_{>0}$ (\emph{i.e.}, the word consisting of $m$ consecutive occurrences of $a$), we have $\binom{w}{a^m}=\binom{|w|_a}{m}$. Indeed, there are $\binom{|w|_a}{m}$ different ways to choose $m$ distinct occurrences of $a$ in $w$, each corresponding to one occurence of $a^m$ in $w$ as a scattered subword.\\
(iii) Since by convention $a^0:=\epsilon$ for any letter $a\in\A$, it is natural to adopt the convention that $\binom{w}{\epsilon}=\binom{|w|_a}{0}=1$ for every finite word $w\in\A^*$.
\end{remark}

We say that a word $w$ is \emph{$c$-balanced} ($c\in\N$) if, for every equally long factors $u,v$ of $w$, and for every letter $a\in\A$, we have $\big||u|_a-|v|_a\big|\leq c$. For example, the word $w=11212$ is $1$-balanced while the word $w'=11122$ is $2$-balanced but not $1$-balanced. The optimal constant of balancedness is termed \emph{imbalance} in \cite{AndPhD,And21,AV22,AV23,AV24}. An infinite word is said to be \emph{unbalanced} if it is not $c$-balanced for any $c\in\N$.

Given two alphabets $\A$ and $\B$, a substitution $\sigma:\A\to\B^*$ is a letter-to-word application, which we extend to a morphism over $\A^*$ (seen as the free monoid for the concatenation). For instance, if $\sigma:\{0,1\}\to\{0,1\}^*$ is defined by $\sigma:0\mapsto 01,\,1\mapsto 0$, then $\sigma(010)=\sigma(0)\sigma(1)\sigma(0)=01001$. More generally, substitutions are also extended over the set of infinite words in the following standard way: if $w=a_0a_1a_2a_3\ldots$ where each $a_i$ is a letter from the alphabet $\A$, then $\sigma(w):=\sigma(a_0)\sigma(a_1)\sigma(a_2)\sigma(a_3)\ldots$.

\bigskip

\noindent\textbf{Complexities.} Two finite words $u,v\in\A^*$ are said to be \emph{abelian equivalent} (resp. \emph{$k$-binomially equivalent}, where $k\in\N_{>0}$ is a parameter) when, for every letter $a\in\A$, $|u|_a=|v|_a$ (resp. when, for every finite word $x\in\A^{\leq k}$, $\binom{u}{x}=\binom{v}{x}$). In such cases, we write $u\sim_{\text{ab}}v$ (resp. $u\sim_{k}v$). For example, if $u=1212221$ and $v=2112212$, then $u\sim_{\text{ab}}v$ and $u\sim_{2}v$, but $u\nsim_{3}v$.

\begin{remark}\label{rk:k-bin}
It follows readily from the above definitions that:\\
(i) If two finite words are abelian (resp. $k$-binomially) equivalent, then they have the same length.\\
(ii) Two $(k+1)$-binomially equivalent words are also $k$-binomially equivalent.\\
(iii) Since for any letter $a\in\A$, $\binom{u}{a}=|u|_a$, two words are $1$-binomially equivalent if and only if they are abelian equivalent.\\
(iv) Two finite words $u$ and $v$ of lengths $|u|,|v|\leq k$ are $k$-binomially equivalent if and only if they are equal.
\end{remark}

These binary relations are equivalence relations. In particular, they partition the language of a word into abelian and $k$-binomial classes, respectively.

The \emph{subword complexity} (resp. \emph{abelian complexity} and \emph{$k$-binomial complexity}) of a word $w$ is the function $p_w:\N\to\N$ (resp. $\rho_w^{ab}:\N\to\N$ and $b_w^k:\N\to\N$) that counts, for each integer $n\in\N$, the number of distinct factors (resp. abelian classes and $k$-binomial classes) of $w$ of length\footnote{
Since two factors are abelian (resp. $k$-binomially) equivalent only if they are of the same length, each abelian (resp. $k$-binomial) class contains only equally long factors. The common length of the factors of a given abelian (resp. $k$-binomial) class is referred to as the length of the class. 
} $n$:
\[
    \begin{array}{rcclrcclrccl}
        p_w:&\N&\rightarrow&\N & \rho_w^{ab}:&\N&\rightarrow&\N & b_w^k:&\N&\rightarrow&\N\\
        & n&\mapsto&\#\big(\lan_n(w)\big) && n&\mapsto&\#\big(\lan_n(w)/_{\abeq}\big) && n&\mapsto&\#\big(\lan_n(w)/_{\sim_k}\big)
    \end{array}
\]

Remark~\ref{rk:k-bin} has the following immediate consequence: for every infinite word $w\in\A^\N$, and every integer $k\in\N_{>0}$,
\begin{equation}\label{eq:scale}
    \rho_w^{ab}=b_w^1 \;\leq\; b_w^k \;\leq\; b_w^{k+1} \;\leq\; p_w.
\end{equation}
In particular, if the $k$-binomial complexity of a word coincides with its subword complexity, then so does its $l$-binomial complexity for every $l\geq k$.

Another consequence of Remark~\ref{rk:k-bin} is that $b_w^k$ converges pointwise to $p_w$ as $k\to\infty$. Consequently, and as we announced it in the introduction, the $k$-binomial complexities form a scale between the abelian complexity and the subword complexity.

% % % % % % % % % % % % % % %
\subsection{Classes of words}

\noindent\textbf{Sturmian and quasi-Sturmian words.} An infinite word is called \emph{Sturmian} if its subword complexity is given by $n\mapsto n+1$. The most famous example of a Sturmian word is the Fibonacci word
\[
    w_{\mathtt{fibo}}=010 01 010 010 01 010 01 010 010 01 010 010 01 010 01 010 010 01 010\ldots
\]
The set of Sturmian words admits many characterizations. For example, an infinite word $w$ is Sturmian if and only if it is binary, $1$-balanced, and non-eventually periodic \cite{MH40}.

A consequence of Morse and Hedlund's Theorem---which states that an infinite word $w$ is eventually periodic if and only if there exists $m\in\N$ such that $p_w(m)\leq m$ \cite{MH38}---is that Sturmian words are exactly the (binary) non-eventually periodic words with minimal complexity. Moreover, the minimal subword complexity of a $d$-ary non-eventually periodic word is given by $n\in\N_{>0}\mapsto n+(d-1)$ (by $d$-ary, we require that $d$ distinct letters appear in the word). Ferenczi and Mauduit provided a characterization of such words in \cite[Lemma~1 and Lemma~4]{FM97}.

More generally, an infinite word $w$ is called \emph{quasi-Sturmian} if there exist two integers $n_0,k_0$ such that $p_w(n)=n+k_0$ for every $n\geq n_0$. These words have been characterized by Cassaigne in \cite[Proposition~8]{Cas97}.

\bigskip

\noindent\textbf{Words with $1$-balanced binary projections.} Let $w$ be a finite or infinite word written over the alphabet $\A$. Its \emph{projection onto a subalphabet $\B\subset\A$}, denoted $\pi_\B(w)$, is obtained by erasing from $w$ all the letters $a\in\A\setminus\B$. Formally, $\pi_\B:\A\to\B^*$ is the substitution defined by $a\in\A\setminus\B\mapsto\epsilon$ and $b\in\B\mapsto b$. For instance, if $w=aabaca$, then $\pi_{a,b}(w)=aabaa$. We say that $w$ has \emph{$1$-balanced binary projections} when all its binary projections are $1$-balanced, \emph{i.e.}, when for every pair of distinct letters $a,b\in\A$, the projection $\pi_{a,b}(w)$ is $1$-balanced.

Words with $1$-balanced binary projections generalize the \emph{ternary words with Sturmian erasures} introduced in \cite{DGK04}. As we will show later, $1$-balanced words, $d$-ary words with subword complexity $n\in\N_{>0}\mapsto n+(d-1)$, and hypercubic billiard words all have $1$-balanced binary projections (see Lemmas~\ref{lemma:c_balanced}, \ref{lemma:minimal_complexityb}, \ref{lemma:projection_billiard} and \ref{lemma:balance_square_billiard}).

\bigskip

\noindent\textbf{Hypercubic billiard words.} In dimension $d\geq 1$, a \emph{hypercubic billiard word} is an infinite $d$-ary word that encodes the sequence of hyperfaces (\emph{i.e.}, $(d-1)$-dimensional faces) successively hit by a billiard ball moving in the unit hypercube of $\R^d$, where parallel hyperfaces are labeled by the same letter (see Figure~\ref{fig}). In what follow, the parameter $\theta$ denotes the initial momentum of the ball and the parameter $x$ denotes its initial position.

\begin{figure}[!h]
\begin{center}
    \includegraphics[scale=0.3]{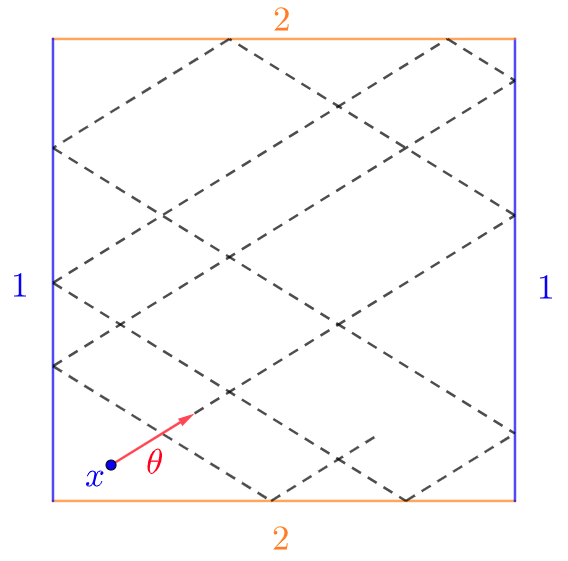}
    \caption{\label{fig} The ball, initially located in $x$ with a momentum $\theta$, generates the infinite word $w=1211212112...$}
\end{center}
\end{figure}

It is well known that a square billiard word generated by a momentum $\theta$ whose coordinates are rationally independent is a Sturmian word \cite{MH40}. More generally, any square billiard word is $1$-balanced. Thus, since any projection of a hypercubic billiard word is itself a lower-dimensional hypercubic billiard word, it follows that all the binary projections of a hypercubic billiard words are $1$-balanced (see Lemmas~\ref{lemma:projection_billiard} and \ref{lemma:balance_square_billiard}).

\bigskip

\noindent\textbf{Coloring of words.} Given two infinite words $w_0\in\A^\N$ and $w_1\in\B^\N$ written over disjoint alphabets ($\A\cap\B=\emptyset$), we define, for every letter $a\in\A$, a new infinite word $\colour(w_0,a,w_1)\in(\A\sqcup\B\setminus\{a\})^\N$ obtained by \emph{coloring the letter $a$ in $w_0$ with $w_1$}. Formally, we set
\[
    \colour(w_0,a,w_1)[n]:=\begin{cases} w_0[n] & \text{if $w_0[n]\neq a$,} \\ w_1[k-1] & \text{if $w_0[n]=a$ and $\big|w_0[0:n]\big|_a=k$.} \end{cases}
\]
For example, the coloring of the letter $\mathtt{a}$ in the Fibonacci word $w_0\in\{\mathtt{0,a}\}^\N$ by the Fibonacci word $w_1\in\{\mathtt{1},\mathtt{2}\}^\N$ itself is obtained as follows:
\[
    \begin{array}{rcl}
        w_0 &=& \mathtt{0a0 0a 0a0 0a0 0a 0a0 0a 0a0 0a0 0a 0a0 0a0 0a 0a0 0a 0a0 0a0 0a 0a0}\ldots\\
        w_1 &=& \mathtt{\aw 1\aw \aw 2 \aw 1\aw \aw 1\aw \aw 2 \aw 1\aw \aw 2 \aw 1\aw \aw 1\aw \aw 2 \aw 1\aw \aw 1\aw \aw 2 \aw 1\aw \aw 2 \aw 1 \aw \aw 1\aw \aw 2 \aw 1\aw}\ldots\\
        \colour(w_0,\mathtt{a},w_1) &=& \mathtt{010 02 010 010 02 010 02 010 010 02 010 010 02 010 02 010 010 02 010}\ldots
    \end{array}
\]

In the literature, colorings of words are used, for example, to characterize non-eventually periodic $1$-balanced words. Indeed, any such word can be obtained as the coloring of a Sturmian word by one or two words with \emph{constant gaps} \cite{Gra73,Hub00}. More recently, Dvo{\v{r}}{\'a}kov{\'a}, Mas{\'a}kov{\'a}, Ma{\v{s}}ek, and Pelantov{\'a} introduced a new class of ternary and quaternary words obtained by coloring a Sturmian word with one or two Sturmian words \cite{DMP23,DMP24}. The word $\colour(w_0,\mathtt{a},w_1)$ defined above belongs to this class. These words share several properties with cubic billiard words, such as quadratic complexity \cite{AMST94a, AMST94b, Bed03} (interested readers may also consult \cite{Bar95,Bed09} for a generalization of this result to arbitrary dimension $d$), $2$-balancedness \cite{Vui03}, and an eventually constant abelian complexity, equal to $4$ \cite{AV23}.

More generally, we define \emph{$d$-ary Sturmian colorings} iteratively as follows:
\begin{itemize}
	\vspace{-0.2cm}\item[-] $1$-ary Sturmian colorings are the infinite constant words (\emph{e.g.}, $0000000\ldots$),
	\vspace{-0.2cm}\item[-] a $d$-ary word is a Sturmian coloring if it can be written as $\colour(w_0,a,w_1)$, where $w_0\in\A^\N$ is a $(d-1)$-ary Sturmian coloring, $a\in\A$, and $w_1$ is a Sturmian word over a binary alphabet $\B$ disjoint from $\A$ (\emph{i.e.}, $\A\cap\B=\emptyset$).
\end{itemize}
\vspace{-0.2cm} It follows from this definition that binary Sturmian colorings are exactly Sturmian words.

% % % % % % % % % % % % % % % % % % % % % % % % % % % % % % % % % % % % % % % % % %
\subsection{Comparison between the four classes of words in Theorem~\ref{th:main}.}

The four classes of words this article focuses on are widely studied in the literature and are presented as generalizations of Sturmian words. Note that none of these classes is contained in another. Indeed, if we set
\[
	\begin{array}{lcl}
		C_1 &:=& \{\text{$d$-ary $1$-balanced words, for $d\geq 3$}\},\\
		C_2 &:=& \{\text{words of subword complexity $n\in\N_{>0}\mapsto n+(d-1)$, for $d\geq 3$}\},\\
		C_3 &:=& \{\text{$d$-ary hypercubic billiard words, for $d\geq 3$}\},\\
		C_4 &:=& \{\text{$d$-ary Sturmian colorings, for $d\geq 3$}\},
	\end{array}
\]
then:
\begin{itemize}
	\item[-] $C_1\nsubseteq C_2\cup C_3\cup C_4$. Indeed, $w:=3(12)^\omega$ belongs to $C_1$ but not to $C_2$ (its subword complexity is bounded), nor to $C_3$ or $C_4$ since it is not recurrent.
	\item[-] $C_2\nsubseteq C_1\cup C_3\cup C_4$. Let $w:=3\cdot S(w_0)=3 21 12 121 121 12\ldots$, where $w_0\in\{1,2\}^\N$ is the Fibonacci word and $S$ denotes the \emph{shift operator}, \emph{i.e.}, the operator that erases the first letter of an infinite word. Then $w$ belongs to $C_2$ but not to $C_1$ (it contains both factors $32$ and $11$), nor to $C_3$ or $C_4$ since it is not recurrent.
	\item[-] $C_3\nsubseteq C_1\cup C_2\cup C_4$. Let $w$ be a cubic billiard word with rationally independent letter frequencies whose inverses are also rationally independent (note that this is the generic situation). Then its subword complexity is given by $n\mapsto n^2+n+1$ \cite{Bar95} (see also \cite[Theorem~6]{Bed03}), which prevents it from belonging to $C_2$ or $C_4$ (the subword complexity of a ternary Sturmian coloring is bounded by $\alpha n^2(1+o(1))$, where $\alpha\in(0,1)$ denotes the frequency of the colored letter in the original Sturmian word \cite[Theorem~9 and Remark~10]{DMP24}). Furthermore, $w$ is not $1$-balanced \cite[Theorem~2]{AV22}. 
	\item[-] $C_4\nsubseteq C_1\cup C_2\cup C_3$. Let $w$ be a ternary Sturmian coloring with rationally independent letter frequencies whose inverses are also rationally independent (this is again the generic situation). Then its subword complexity satisfies $p_w(n)\sim_{n\to\infty}\alpha n^2$, where $\alpha$ again denotes the frequency of the colored letter in the original Sturmian word \cite[Theorem~16 and Remark~10]{DMP24}, which prevents it from belonging to $C_2$ or $C_3$. Moreover, $w$ is not $1$-balanced \cite[Theorem~9]{DMP24}.
\end{itemize}

However, these classes are not disjoint. An interesting example is the following. Let $w_0\in\{0,1\}^\N$ be the Fibonacci word (or any other Sturmian word), and set $w_1:=(23)^\omega$. One can show that the colored word $w:=\colour(w_0,0,w_1)$ belongs to $C_1\cap C_3\cap C_4$. In other words, $w$ is simultaneously a $1$-balanced ternary word, a cubic billiard word, and the coloring of a Sturmian word by another Sturmian word (of course, for a coloring different from the one defining $w$).

%%%%%%%%%%%%%%%%%%%%%%%%%%%%%%%%%%%%%%
\section{Main results and proof strategies}\label{sec:main}

In this section we present the proof strategy and the intermadiate results required for the proof of Theorem~\ref{th:main}. We also discuss the optimality of each assertion.

% % % % % % % % % % % % % % % % % % %
\subsection{Binary $1$-balanced words}

We begin by generalizing, to the broader class of binary $1$-balanced words, the result of Rigo and Salimov concerning the $2$-binomial complexity of Sturmian words \cite[Theorem~7]{RS15}.

\begin{proposition}\label{prop:binary}
If $w$ is a binary $1$-balanced word, then its $2$-binomial complexity coincides with its subword complexity.
\end{proposition}

According to the terminology introduced by Morse and Hedlund in \cite{MH40}, this generalization also applies to \emph{periodic} and \emph{skew} Sturmian words. As described in their work, these words have a combinatorial structure very similar to that of Sturmian words. Thus, this proposition represents only a slight extension of Rigo and Salimov's result. However, since it is the first step toward the proof of Theorem~\ref{th:main}---and for the sake of self-containedness---we provide a detailed proof in Section~\ref{sec:binary}.

The converse of Proposition~\ref{prop:binary} is false. Indeed, it is easy to see that the eventually constant binary word $w=11222222\ldots$ is $2$-balanced but not $1$-balanced, and yet its $2$-binomial complexity coincides with its subword complexity (\emph{cf} Proposition~\ref{prop:1bin}). More generally, for every integer $m\geq 2$, the eventually constant word $1^m 222222\ldots$ is $m$-balanced but not $(m-1)$-balanced, and its $2$-binomial complexity coincides with its subword complexity.

Furthermore, Proposition~\ref{prop:binary} does not admit any straightforward generalization. Indeed, the Thue-Morse word $w_{\mathtt{tm}}=0110100110010110\ldots$ is binary and $2$-balanced, but its $2$-binomial complexity is strictly smaller than its subword complexity (it sufficies to notice that $0110,1001\in\lan(w_{\mathtt{tm}})$ and $0110\sim_2 1001$). This shows that the assumption ``$w$ is $1$-balanced'' in Proposition~\ref{prop:binary} cannot be weakened to ``$w$ is $2$-balanced''. More generally, since for every $k\in\N_{>0}$, the $k$-binomial complexity of the Thue-Morse word never coincides with its subword complexity (indeed, the $k$-binomial complexity of any fixed point of a \emph{Parikh-constant substitution} is bounded, see \cite[Theorem~13]{RS15}), it follows that even the weaker statement---``if $w$ is a binary $2$-balanced word, then there exists an integer $k$ such that its $k$-binomial complexity coincides with its subword complexity''---does not hold.

% % % % % % % % % % % % % % % % % % % % % % % % % % % %
\subsection{Words with binary $1$-balanced projections}

The following key lemma provides a sufficient condition for a word to have its $2$-binomial complexity equal to its subword complexity.

\begin{lemma}\label{lemma:binary_projections}
Let $d\geq 2$ and $w\in\{1,\ldots,d\}^\N$. If all the binary projections of $w$ are $1$-balanced, then its $2$-binomial complexity is equal to its subword complexity.
\end{lemma}

This result is a special instance of the following more general lemma.

\begin{lemma}\label{lemma:projections_kbin}
Let $d\geq 2$, $k\in\N_{>0}$ and $w\in\{1,\ldots,d\}^\N$. If, for every pair of distinct letters $i,j\in\{1,\ldots,d\}$, $b_{\pi_{i,j}(w)}^k=p_{\pi_{i,j}(w)}$, then the $k$-binomial complexity of $w$ coincides with its subword complexity.
\end{lemma}

\begin{proof}[Proof of Lemma~\ref{lemma:binary_projections}]
Thanks to Proposition~\ref{prop:binary}, if all the binary projections of a given word $w$ are $1$-balanced, then the $2$-binomial complexity of each binary projection coincides with its subword complexity. As a result, Lemma~\ref{lemma:binary_projections} follows from Lemma~\ref{lemma:projections_kbin} applied with $k=2$.
\end{proof}

The proof of Lemma~\ref{lemma:projections_kbin} is given in Section~\ref{sec:binary_projections}, where we also establish statements (i)--(iii) of Theorem~\ref{th:main} by showing that all the binary projections of any of the following words are $1$-balanced:
\begin{itemize}
    \vspace{-0.2cm}\item[-] $d$-ary $1$-balanced words (Lemma~\ref{lemma:c_balanced}),
    \vspace{-0.2cm}\item[-] $d$-ary words with subword complexity $n\in\N_{>0}\mapsto n+(d-1)$ (Lemma~\ref{lemma:minimal_complexityb}),
    \vspace{-0.2cm}\item[-] hypercubic billiard words in dimension $d$ (Lemmas~\ref{lemma:projection_billiard} and \ref{lemma:balance_square_billiard}).
\end{itemize}

Note that statement (ii) of Theorem~\ref{th:main} (words with subword complexity $n\in\N_{>0}\mapsto n+(d-1)$) cannot be extended to all quasi-Sturmian words. Indeed, consider the substitution $\sigma:1\mapsto 1221, \, 2\mapsto 2112$. Then, according to the characterization by Cassaigne (see \cite[Proposition~8]{Cas97}), for every Sturmian word $w_0\in\{1,2\}^\N$, the image word $w:=\sigma(w_0)$ is a quasi-Sturmian word. Since $1221\sim_2 2112$, it is clear that $b_w^2\neq p_w$.

It is worth mentioning that the result for hypercubic billiard words holds for \textbf{every} such word: no additional assumption on the initial momentum $\theta$ of the ball is required. This contrasts with the expression of their subword complexity \cite{Bar95,Bed09}, which only holds under additional hypotheses on the momentum $\theta$. Thus, even when no explicit expression for the subword complexity is known, the $2$-binomial complexity of these words is still equal to their subword complexity.

\medskip

The converses of Lemmas~\ref{lemma:binary_projections} and \ref{lemma:projections_kbin} do not hold: the Tribonacci word 
\[
    w_{\texttt{tribo}}=1213121 121312 1213121 1213 1213121 121312 1213121 1213121\ldots
\]
provides a counterexample. For instance, its $2$-binomial complexity coincides with its subword complexity; however, none of its three binary projections is $1$-balanced. Indeed, one can check that
\[
    \begin{array}{c}
        u=11211211211211,\,v=21211211211212\in\lan_{14}\big(\pi_{1,2}(w_{\texttt{tribo}})\big) \hspace{0.5cm} \text{and} \hspace{0.5cm} |u|_1-|v|_1=2,\\
        \\
        u=1111,\,v=3113\in\lan_{4}\big(\pi_{1,3}(w_{\texttt{tribo}})\big) \hspace{0.5cm} \text{and} \hspace{0.5cm} |u|_1-|v|_1=2,\\
        \\
        u=22322322322322,\,v=32322322322323\in\lan_{14}\big(\pi_{2,3}(w_{\texttt{tribo}})\big) \hspace{0.5cm} \text{and} \hspace{0.5cm} |u|_2-|v|_2=2,
    \end{array}
\]
which proves that the converse of Lemma~\ref{lemma:binary_projections} is false. More generally, none of the three binary projections of the Tribonacci word has its $2$-binomial complexity equal to its subword complexity. Indeed, one can check that
\[
    \begin{array}{c}
        u=2112112112112112,\,v=1212112112112121\in\lan_{16}\big(\pi_{1,2}(w_{\texttt{tribo}})\big) \hspace{0.5cm} \text{and} \hspace{0.5cm} u\sim_2 v,\\
        \\
        u=311113,\,v=131131\in\lan_{6}\big(\pi_{1,3}(w_{\texttt{tribo}})\big) \hspace{0.5cm} \text{and} \hspace{0.5cm} u\sim_2 v,\\
        \\
        u=3223223223223223,\,v=2323223223223232\in\lan_{16}\big(\pi_{2,3}(w_{\texttt{tribo}})\big) \hspace{0.5cm} \text{and} \hspace{0.5cm} u\sim_2 v,
    \end{array}
\]
which shows that the converse of Lemma~\ref{lemma:projections_kbin} is false.

% % % % % % % % % % % % % % % % % 
\subsection{Stability by coloring}

We now present our second key lemma, which establishes a stability result under coloring for words whose $k$-binomial complexity coincides with their subword complexity. Its proof is given in Section~\ref{sec:stability}.

\begin{lemma}\label{lemma:stability}
Let $w_0\in\A^\N$ and $w_1\in\B^\N$ be two infinite words written over disjoint alphabets. If there exists $k\in\N_{>0}$ such that $b_{w_0}^k=p_{w_0}$ and $b_{w_1}^k=p_{w_1}$, then, for every letter $a\in\A$, the $k$-binomial complexity of $\colour(w_0,a,w_1)$ coincides with its subword complexity.
\end{lemma}

As an immediate consequence, Lemma~\ref{lemma:stability} furnishes an inductive method to construct words whose $k$-binomial complexity is equal to their subword complexity over increasingly larger alphabets. In particular, since the $2$-binomial complexity of Sturmian words coincides with their subword complexity, statement (iv) of Theorem~\ref{th:main} follows from Lemma~\ref{lemma:stability} via a straightforward induction. Note that this statement cannot be proved from Lemma~\ref{lemma:binary_projections} or Lemma~\ref{lemma:projections_kbin}. Indeed, the binary projection
\[
	\pi_{0,2}(w)=00 02 00 00 02 00 02 00 00 02 00 00 02 00 02 00 00 02 00\ldots
\]
of the infinite word $w:=\colour(w_0,a,w_1)$ defined in Section~\ref{sec:def}---obtained by coloring the letter $a$ in the Fibonacci word $w_0\in\{0,a\}^\N$ with the Fibonacci word $w_1\in\{1,2\}^\N$---is not $1$-balanced (it contains the factors $\mathtt{00000}$ and $\mathtt{20002}$). Moreover, since $2000002$ and $0200020$ belong to $\lan_7(\pi_{0,2}(w))$ and satisfy $2000002\sim_2 0200020$, it follows that the $2$-binomial complexity of $\pi_{0,2}(w)$ also differs from its subword complexity.

The converse of Lemma~\ref{lemma:stability} is trivially true: any word $w$ whose $k$-binomial complexity coincides with its subword complexity can be expressed as $\colour(w_0,a,w_1)$, where $w_0$ and $w_1$ are two infinite words over disjoint alphabets such that $b_{w_0}^k=p_{w_0}$ and $b_{w_1}^k=p_{w_1}$. Indeed, the simple choice $w_0=111111\ldots$, $a=1$, and $w_1=w$ suffices. However, the more interesting converse statement---requiring that both $w_0$ and $w_1$ are defined over alphabets of \textbf{strictly smaller} size than $p_w(1)$ (the number of distinct letters occurring in $w$)---does not hold. The Tribonacci word is again a counterexample to such a statement. Indeed, its $2$-binomial complexity coincides with its subword complexity but, as already seen, none of its three binary projections satisfies this property. In particular, this shows that the Tribonacci word cannot be obtained as the coloring of two binary words such that both of them have their $2$-binomial complexity equal to their subword complexity.

%%%%%%%%%%%%%%%%%%%%%%%%%%%%%%%%%%%%%%%%%%%%%%%%%%%%%
\section{Proof of Proposition~\ref{prop:binary}}\label{sec:binary}

This section is devoted to the proof of Proposition~\ref{prop:binary}: the $2$-binomial complexity of any binary $1$-balanced word coincides with its subword complexity. As already mentioned, the case of Sturmian words was treated by Rigo and Salimov in \cite{RS15}. Following the terminology of Morse and Hedlund \cite{MH40}, it remains only to consider the cases of \emph{periodic} and \emph{skew} Sturmian words. Morse and Hedlund showed that these words have a combinatorial structure closely similar to that of Sturmian words, and the proof of Rigo and Salimov can be extended straightforwardly to these two cases. However, for the reader's convenience and to ensure this article remains self-contained, we provide a complete proof of Proposition~\ref{prop:binary}. Although our proof is, in spirit, essentially the same as that of Rigo and Salimov, we take this opportunity to present it in a more ``elementary form'' -- that is, requiring no prerequisites beyond the definition of $1$-balanced words.

\begin{proof}[Proof of Proposition~\ref{prop:binary}]
We proceed by contrapositive and consider an infinite word $w\in\{1,2\}^\N$ such that there exists $m\in\N_{>0}$ for which $b_w^2(m)<p_w(m)$, \emph{i.e.}, such that there exists two factors $u,v\in\lan_m(w)$ for which $u\neq v$ and $u\sim_{2\text{-bin}}v$. Our goal is to prove that such word $w$ is not $1$-balanced.

Since $u\neq v$, there exists $p,u',v',s\in\lan(w)$ with $u',v'\neq\epsilon$ such that $u=pu's$, $v=pv's$ and the first (resp. the last) letter of $u'$ differs from that of $v'$. We claim that $u'$ and $v'$ are $2$-binomially equivalent (this result is called the \emph{cancellation property} in \cite[Lemma~2.2]{RSW24}). Indeed, for every $i\in\{1,2\}$ one immediatly verifies that
\[
    \binom{u}{i}=\binom{pu's}{i}=\binom{p}{i}+\binom{u'}{i}+\binom{s}{i} \hspace{0.5cm}\text{ and }\hspace{0.5cm} \binom{v}{i}=\binom{pv's}{i}=\binom{p}{i}+\binom{v'}{i}+\binom{s}{i},
\]
and the equality $\binom{u}{i}=\binom{v}{i}$ yields $\binom{u'}{i}=\binom{v'}{i}$. Moreover, for every $i,j\in\{1,2\}$ (possibly equal), we have:
\[
    \begin{array}{rcl}
        \ds\binom{u}{ij}=\binom{pu's}{ij} &=& \ds\binom{p}{ij}+\binom{p}{i}\binom{u's}{j}+\binom{u's}{ij}\\
        \\
        &=&\ds\binom{p}{ij}+\binom{p}{i}\left[\binom{u'}{j}+\binom{s}{j}\right]+\binom{u'}{ij}+\binom{u'}{i}\binom{s}{j}+\binom{s}{ij}.
    \end{array}
\]
Indeed, $ij$ is a scattered factor of $pu's$ if and only if one of the following holds: 1) $ij$ is a scattered factor of $p$, 2) $ij$ is a scattered factor of $u's$, 3) $i$ occurs in $p$ and $j$ occurs in $u's$. Therefore, the number of occurrences of $ij$ in $pu's$ as a scattered factor equals its number of occurrences in $p$ and $u's$ as a scattered factor, plus the product of the number of occurrences of $i$ in $p$ and the number of occurrences of $j$ in $u's$. This relation is a particular case of a more general identity that is the analogue, for binomial coefficients of words, of the Chu--Vandermonde identity; see for example \cite[Proposition~1]{RS15}. The binomial coefficient $\binom{u's}{ij}$ has been computed similarly. Symmetrically, we also have:
\[
    \binom{v}{ij}=\binom{p}{ij}+\binom{p}{i}\left[\binom{v'}{j}+\binom{s}{j}\right]+\binom{v'}{ij}+\binom{v'}{i}\binom{s}{j}+\binom{s}{ij}.
\]
Since $\binom{u}{ij}=\binom{v}{ij}$, $\binom{u'}{i}=\binom{v'}{i}$, and $\binom{u'}{j}=\binom{v'}{j}$, the two previous relations yield $\binom{u'}{ij}=\binom{v'}{ij}$. We have proved that, for every $x\in\{1,2\}^{\leq 2}$, $\binom{u'}{x}=\binom{v'}{x}$, \emph{i.e.}, $u'$ and $v'$ are $2$-binomially equivalent. In particular, note that $|u'|=|v'|\geq 2$. Indeed, they must be equally long because they are $2$-binomially equivalent, and they cannot be reduced to a single letter since their first letters differ and they have the same number of occurrences of each letter.

Up to renaming the letters if necessary, it suffices to consider the following two cases (recall that $|u'|=|v'|\geq 2$):

\begin{itemize}
    \vspace{-0.3cm}\item[-] Case 1: $u'=1u''1$ and $v'=2v''2$,
    \vspace{-0.3cm}\item[-] Case 2: $u'=1u''2$ and $v'=2v''1$.
\end{itemize}
\vspace{-0.3cm}
In the first case, we have $|u''|_2-|v''|_2=|u'|_2-(|v'|_2-2)=2$, and $w$ is not $1$-balanced. We now consider the second case. Since $m:=|u'|_1=|v'|_1\geq 1$, there exist $2m$ integers $k_1,\ldots,k_m,l_1,\ldots,l_m\in\N$ with $k_m\geq 1$ and $l_1\geq 1$, such that
\[
    u'=12^{k_1}12^{k_2}1\ldots 12^{k_m} \hspace{0.5cm}\text{ and }\hspace{0.5cm} v'=2^{l_1}12^{l_2}1\ldots 12^{l_m}1.
\]
Using this decomposition of $u'$ and $v'$, we can compute $\binom{u'}{12}$ and $\binom{v'}{12}$ as follows. Since
\[
    \begin{array}{rcl}
        \ds\binom{u'}{12} = \binom{1 2^{k_1} 1 2^{k_2} 1\ldots 1 2^{k_m}}{12} &=& \ds\binom{1}{1}\binom{2^{k_1} 1 2^{k_2} 1\ldots 1 2^{k_m}}{2}+\binom{2^{k_1} 1 2^{k_2} 1\ldots 1 2^{k_m}}{12}\\
        \\
        &=& \ds\Big(\sum\limits_{i=1}^m k_i\Big)+\binom{ 1 2^{k_2} 1\ldots 1 2^{k_m}}{12},
    \end{array}
\]
we obtain inductively:
\[
    \binom{u'}{12}=\Big(\sum\limits_{i=1}^{m}k_i\Big)+\Big(\sum\limits_{i=2}^{m}k_i\Big)+\ldots+\Big(\sum\limits_{i=m}^{m}k_i\Big),
\]
that we rewrite:
\[
    \binom{u'}{12}=\ds\Big(\sum\limits_{i=1}^{m-1}k_i\Big)+\Big(\sum\limits_{i=2}^{m-1}k_i\Big)+\ldots+\Big(\sum\limits_{i=m-1}^{m-1}k_i\Big)+mk_m \;\;=\;\; \sum\limits_{j=1}^{m-1}\Big(\sum\limits_{i=j}^{m-1} k_i\Big)+mk_m.
\]
Similarly:
\[
    \binom{v'}{12}=\Big(\sum\limits_{i=2}^{m}l_i\Big)+\Big(\sum\limits_{i=3}^{m}l_i\Big)+\ldots+\Big(\sum\limits_{i=m}^{m}l_i\Big) =\sum\limits_{j=2}^{m}\Big(\sum\limits_{i=j}^m l_i\Big)=\sum\limits_{j=1}^{m-1}\Big(\sum\limits_{i=j+1}^m l_i\Big),
\]
and we eventually obtain:
\[
    \begin{array}{rcl}
        \ds\binom{u'}{12}=\binom{v'}{12} &\Longrightarrow& \ds mk_m=\sum\limits_{j=1}^{m-1}\left[\Big(\sum\limits_{i=j+1}^m l_i\Big)-\Big(\sum\limits_{i=j}^{m-1} k_i\Big)\right].
    \end{array}
\]
We distinguish two subcases:
\begin{itemize}
    \vspace{-0.3cm}
    \item[-] Case 2a: For every $j\in\{1,\ldots,m-1\}$, $\ds\Big(\sum\limits_{i=j+1}^m l_i\Big)-\Big(\sum\limits_{i=j}^{m-1} k_i\Big)\leq 1$.
    \vspace{-0.3cm}
    \item[-] Case 2b: There exists $j\in\{1,\ldots,m-1\}$ such that $\ds\Big(\sum\limits_{i=j+1}^m l_i\Big)-\Big(\sum\limits_{i=j}^{m-1} k_i\Big)\geq 2$.
\end{itemize}
\vspace{-0.3cm}
Case 2a cannot occur. Indeed, since $k_m\geq 1$, we would have the following contradiction:
\[
    m\leq mk_m=\sum\limits_{j=1}^{m-1}\left[\Big(\sum\limits_{i=j+1}^m l_i\Big)-\Big(\sum\limits_{i=j}^{m-1} k_i\Big)\right]\leq \sum\limits_{j=1}^{m-1} 1=m-1.
\]
In case 2b, consider the factors $x=12^{k_j}1\ldots 12^{k_{m-1}}1\in\lan(u')$ and $y=2^{l_{j+1}}1\ldots 12^{l_m}\in\lan(v')$. These satisfy $|x|_1=|y|_1+2$ and $|y|_2\geq |x|_2+2$, hence $|y|\geq |x|$. Therefore, any factor $y'$ of $y$ of length $|x|$ satisfies $|x|_1-|y'|_1\geq|x|_1-|y|_1=2$, which shows that $w$ is not $1$-balanced.
\end{proof}

%%%%%%%%%%%%%%%%%%%%%%%%%%%%%%%%%%%%%%%%%%%%%%%%%%%%%%%%%%%%%%%%%%%%%%%%%%%%%%%%%%
\section{Proof of Lemma~\ref{lemma:projections_kbin} and of Theorem~\ref{th:main}, assertions (i)--(iii)}\label{sec:binary_projections}

We begin by stating and proving (since the proof is short) a known reconstruction lemma (see \cite[Lemma~6.2.19]{Loth97}).

\begin{lemma}\label{lemma:reconstruction}
Let $d\geq 2$ and $u\in\{1,\ldots,d\}^*$. The word $u$ is uniquely determined by the set of its binary projections $\{\pi_{i,j}(u)\;|\;i,j\in\{1,\ldots,d\} \text{ and }i\neq j\}$.
\end{lemma}

\begin{proof}
We argue by contradiction and assume that there exists $v\in\{1,\ldots,d\}^*$ such that $v\neq u$ and yet $\pi_{i,j}(v)=\pi_{i,j}(u)$ for every pair of distinct letters $i,j\in\{1,\ldots,d\}$. Since $v\neq u$, there must exist $p,u',v'\in\{1,\ldots,d\}^*$ and two distinct letters $i,j\in\{1,\ldots,d\}$ such that $u=piu'$ and $v=pjv'$. However, the equality $\pi_{i,j}(v)=\pi_{i,j}(u)$ implies that $\pi_{i,j}(p)i\pi_{i,j}(v')=\pi_{i,j}(p)j\pi_{i,j}(u')$, which leads to the contradiction $i=j$. Therefore, $u$ and $v$ must be equal, completing the proof.
\end{proof}

\begin{proof}[Proof of Lemma~\ref{lemma:projections_kbin}]
Let $w\in\{1,\ldots,d\}^\N$ with $d\geq 2$, and assume that there exists $k\in\N_{>0}$ such that, for every distinct letters $i,j\in\{1,\ldots,d\}$, the $k$-binomial complexity of $\pi_{i,j}(w)$ coincides with its subword complexity. We aim to prove that $w$ fulfills the same property, \emph{i.e.}, that every factor of $w$ is alone in its $k$-binomial equivalence class. Let $u,v\in\lan(w)$ be two factors of $w$ such that $u\sim_{k}v$. Our goal is to prove that $u=v$.

First, for every distinct letters $i,j\in\{1,\ldots,d\}$, and for every finite word $x\in\{i,j\}^{\leq k}\subseteq\{1,\ldots,d\}^{\leq k}$, we have $\binom{u}{x}=\binom{v}{x}$. Moreover, in such a case we also have
\begin{equation}\label{eq:projections}
    \binom{\pi_{i,j}(u)}{x}=\binom{u}{x} \hspace{0.3cm}\text{ and }\hspace{0.3cm} \binom{\pi_{i,j}(v)}{x}=\binom{v}{x}.
\end{equation}
Indeed, the number of occurrences of $x\in\{i,j\}^*$ in $u$ (resp. $v$), as a scattered subword, does not depend on how many letters from $\{1,\ldots,d\}\setminus\{i,j\}$ appear in $u$ (resp. $v$), nor on their positions.

Gathering these relations, we have shown that for every $x\in\{i,j\}^{\leq k}$, $\binom{\pi_{i,j}(u)}{x}=\binom{\pi_{i,j}(v)}{x}$, \emph{i.e.}, $\pi_{i,j}(u)\sim_k\pi_{i,j}(v)$. Since both projections $\pi_{i,j}(u)$ and $\pi_{i,j}(v)$ are factors of $\pi_{i,j}(w)$, and since $b_{\pi_{i,j}(w)}^k=p_{\pi_{i,j}(w)}$, it follows that $\pi_{i,j}(u)=\pi_{i,j}(v)$. As this equality holds for every pair of distinct letters $i,j\in\{1,\ldots,d\}$, Lemma~\ref{lemma:reconstruction} allows us to conclude that $u=v$.
\end{proof}

The proof of Theorem~\ref{th:main}, statement (i) relies on the following lemma.

\begin{lemma}\label{lemma:c_balanced}
Let $d\geq 2$ and $w\in\{1,\ldots,d\}^\N$. If $w$ is $c$-balanced, then for every pair of distinct letters $i,j\in\{1,\ldots,d\}$, the binary projection $\pi_{i,j}(w)$ is also $c$-balanced.
\end{lemma}

\begin{proof}
We proceed by contradiction and assume that there exist two distinct letters $i,j\in\{1,\ldots,d\}$ such that $\pi_{i,j}(w)$ is not $c$-balanced. Then there exist two equally long factors $u,v\in\lan(\pi_{i,j}(w))$ and a letter $a\in\{i,j\}$ such that $|\,|u|_a-|v|_a|\geq c+1$. Moreover, since $\pi_{i,j}(w)$ is a binary word, and since $u$ and $v$ have the same length, we have
\[
    |u|_i-|v|_i=|v|_j-|u|_j,
\]
and the previous inequality holds for both $a=i$ and $a=j$. Without loss of generality, we assume that the letter $i$ has more occurrences in $u$ than in $v$, so (the letter $j$ occurs more in $v$ than $u$ and)
\[
    |u|_i-|v|_i=|v|_j-|u|_j\geq c+1.
\]

We now consider factors $u',v'\in\lan(w)$ of $w$ such that $\pi_{i,j}(u')=u$ and $\pi_{i,j}(v')=v$. If $|u'|\leq |v'|$, then writing $v'=ps$ with $|p|=|u'|$, we have
\[
    |u'|_{i}-|p|_{i}\geq |u'|_i-|v'|_i=|\pi_{i,j}(u')|_i-|\pi_{i,j}(v')|_i\geq c+1,
\]
which contradicts the $c$-balancedness of $w$. Symmetrically, the case $|u'|>|v'|$ leads to a similar contradiction.
\end{proof}

\begin{proof}[Proof of Theorem~\ref{th:main}, (i)]
Lemma~\ref{lemma:c_balanced} implies that all the binary projections of any $d$-ary $1$-balanced word are $1$-balanced. The result then follows immediately from Lemma~\ref{lemma:binary_projections}.
\end{proof}

The proof of Theorem~\ref{th:main}, statement (ii) comes from the following characterization of words with subword complexity $n\in\N_{>0}\mapsto n+(d-1)$. We refer the reader to \cite[Lemma~1 and Lemma~4]{FM97} for a proof of this result.

\begin{lemma}[Ferenczi, Mauduit, 97]\label{lemma:minimal_complexity}
Let $\A$ be a $d$-ary alphabet with $d\geq 3$, and let $w\in\A^\N$ be a word with subword complexity $n\in\N_{>0}\mapsto n+(d-1)$.\\
\emph{(i)} If $w$ is recurrent (\emph{i.e.}, every factor of $w$ occurs infinitely often in $w$), then there exist:
\begin{itemize}
    \vspace{-0.3cm}
    \item[-] a Sturmian word $w_0\in\{1,2\}^\N$,
    \vspace{-0.3cm}
    \item[-] a partition $\A=\mathcal{B}\sqcup\mathcal{C}\sqcup\mathcal{D}$ where $\mathcal{B}=\{b_1,\ldots,b_{N_{\mathcal{B}}}\}$, $\mathcal{C}=\{c_1,\ldots,c_{N_{\mathcal{C}}}\}$ and $\mathcal{D}=\{d_1,\ldots,d_{N_{\mathcal{D}}}\}$, with both $\mathcal{B}\neq\emptyset$ and $\mathcal{C}\sqcup\mathcal{D}\neq\emptyset$,
\end{itemize}
\vspace{-0.2cm}
such that $pw=\sigma(w_0)$, where $\sigma:\{1,2\}\to\A^*$ is the substitution
\[
    1\mapsto b_1\ldots b_{N_{\mathcal{B}}}c_1\ldots c_{N_{\mathcal{C}}}, \hspace{0.5cm} 2\mapsto b_1\ldots b_{N_{\mathcal{B}}}d_1\ldots d_{N_{\mathcal{D}}},
\]
and where $p$ is a (possibly empty) prefix of $\sigma(1)$ or $\sigma(2)$.\\
\emph{(ii)} If $w$ is not recurrent, then there exist $1\leq d'<d$, a $d'$-letter subalphabet $\mathcal{B}\subset\A$, and a recurrent word $w_0\in\mathcal{B}^\N$ with subword complexity $n\in\N_{>0}\mapsto n+(d'-1)$ such that
\[
    w=a_1\ldots a_{d-d'} w_0,
\]
where the letters $a_i$ are the $d-d'$ distinct elements of $\A\setminus\mathcal{B}$.
\end{lemma}

This characterization has the following consequence.

\begin{lemma}\label{lemma:minimal_complexityb}
Let $d\geq 3$ and let $w$ be a $d$-ary word. If $p_w(n)=n+(d-1)$ for every $n\in\N_{>0}$, then all the binary projections of $w$ are $1$-balanced.
\end{lemma}

\begin{proof}
Let $w\in\A^\N$ be a $d$-ary word with subword complexity $p_w(n)=n+(d-1)$ for every $n\geq 1$. We distinguish two cases.

\medskip

\emph{Case 1. The word $w$ is recurrent.} We write $w$ as in Lemma~\ref{lemma:minimal_complexity}, item (i). Our goal is to prove that, for every pair of distinct letters $a,b\in\A$, the projection $\pi_{a,b}(w)$ is $1$-balanced.

First, if $a$ and $b$ belong to the same subalphabet $\mathcal{B}$, $\mathcal{C}$, or $\mathcal{D}$, then $\pi_{a,b}(w)$ is purely periodic, with period $ab$ or $ba$. Indeed, $\pi_{a,b}(\sigma(1))$ and $\pi_{a,b}(\sigma(2))$ are equal to $ab$, $ba$, or the empty word $\epsilon$, depending on the subalphabet (note that $ab$ and $ba$ cannot both occur). In particular, $\pi_{a,b}(w)$ is $1$-balanced.

Secondly, if $a,b\in\mathcal{C}\sqcup\mathcal{D}$ but belong to different subalphabets---say $a\in\mathcal{C}$ and $b\in\mathcal{D}$---then $\pi_{a,b}(\sigma(1))=a$ and $\pi_{a,b}(\sigma(2))=b$. Thus, up to replace the letters $1$ by $a$ and the letters $2$ by $b$ in $w_0$, we have $\pi_{a,b}(w)=w_0$ or $S(w_0)$, where $S$ is the shift operator acting on infinite words ($S(w)[n]=w[n+1]$). Consequently $\pi_{a,b}(w)$ is Sturmian and, in particular, $1$-balanced.

Finally, if $a\in\mathcal{B}$ and $b\in\mathcal{C}\sqcup\mathcal{D}$---say $b\in\mathcal{C}$---then $\pi_{a,b}(\sigma(1))=ab$ and $\pi_{a,b}(\sigma(2))=a$. In this case, $\pi_{a,b}(w)=\tau(w_0)$ or $S(\tau(w_0))$ or $\tau(S(w_0))$, where $\tau:\{1,2\}\to\{a,b\}^*$ is the substitution defined by $\tau(1)=ab$ and $\tau(2)=a$. It is well known that such a substitution maps Sturmian words to Sturmian words, see \cite[Chapter~2, Section~2.3]{Loth02}. Thus, $\pi_{a,b}(w)$ is $1$-balanced.

\medskip

\emph{Case 2. The word $w$ is not recurrent.} According to Lemma~\ref{lemma:minimal_complexity}, we have $w=a_1\ldots a_{d-d'}w_0$, where $w_0\in\B^\N$ is a recurrent word whose subword complexity is $n\in\N_{>0}\mapsto n+(d'-1)$, and where the letters $a_i\in\A\setminus\B$ are distinct and do not occur in $w_0$. Thus, for every pair of distinct letters $a,b\in\A$, $\pi_{a,b}(w)$ is either equal to $ab$ (both $a$ and $b$ belong to $\A\setminus\B$), or $ab^\omega$ ($a\in\A\setminus\B$ and $b\in\B$), or $\pi_{a,b}(w_0)$ (both $a$ and $b$ belong to $\B$). In the first two cases $\pi_{a,b}(w)$ is clearly $1$-balanced, while the third case has already been treated in \emph{Case 1}.
\end{proof}

\begin{proof}[Proof of Theorem~\ref{th:main}, (ii)]
Let $d\geq 2$ and $w$ be a $d$-ary word with subword complexity $n\in\N_{>0}\mapsto n+(d-1)$. If $d=2$, then $w$ is Sturmian and the result is already known. If $d\geq 3$, then according to Lemma~\ref{lemma:minimal_complexityb} all the binary projections of $w$ are $1$-balanced. The result is then an immediate consequence of Lemma~\ref{lemma:binary_projections}.
\end{proof}

We conclude this section with the proof of Theorem~\ref{th:main}, statement (iii). This result relies on the following two well-known lemmas.

\begin{lemma}\label{lemma:projection_billiard}
Let $d\geq 2$, and let $w\in\{1,\ldots,d\}^\N$ be a hypercubic billiard word in dimension $d$. Then, for every subalphabet $\B\subset\{1,\ldots,d\}$, the projection $\pi_\B(w)$ is a hypercubic billiard word in dimension $\#\B$.
\end{lemma}

\begin{proof}[Sketch of proof]
To see this, consider a ball moving inside a $d$-dimensional hypercubic billiard table. The projection of its trajectory onto any hyperface of the table corresponds to the trajectory of a ball moving inside a $(d-1)$-dimensional hypercubic billiard table. Moreover, one can verify that the codings of these two trajectories are related: the second coding is obtained from the first by erasing the letter corresponding to the label of the hyperface onto which the original trajectory is projected. The result then follows by an immediate induction on the dimension $d$.
\end{proof}

The second lemma follows from the characterization of $1$-balanced words by Morse and Hedlund \cite{MH40}. Interested readers may consult \cite[Section~6]{Vui03} and \cite{AV22,AV24} for results concerning the balancedness constants of hypercubic billiard words in arbitrary dimensions.

\begin{lemma}\label{lemma:balance_square_billiard}
Square billiard words are $1$-balanced.
\end{lemma}

\begin{proof}[Sketch of Proof]
The usual \emph{unfolding procedure} shows that square billiard words can equivalently be obtained as \emph{cutting words}. The connection between cutting words, mechanical words, codings of rotations and Sturmian words is fully detailed in \cite[Chapter~2, Section~2.1]{Loth02}, see in particular Theorem~2.1.13 and Lemmas~2.1.14 and 2.1.15 in that reference.
\end{proof}

\begin{proof}[Proof of Theorem~\ref{th:main}, (iii)]
Lemmas~\ref{lemma:projection_billiard} and \ref{lemma:balance_square_billiard} imply that all the binary projections of any hypercubic billiard words are $1$-balanced. The result then follows once more from Lemma~\ref{lemma:binary_projections}.
\end{proof}

%%%%%%%%%%%%%%%%%%%%%%%%%%%%%%%%%%%%%%%%%%%%%%%%%%%%%
\section{Proof of Lemma~\ref{lemma:stability} and Theorem~\ref{th:main}, assertion (iv)}\label{sec:stability}

We now prove Lemma~\ref{lemma:stability}, which asserts that words whose $k$-binomial complexity coincides with their subword complexity remain stable under the coloring process described in Section~\ref{sec:def}. As mentioned in Section~\ref{sec:main}, since the $2$-binomial complexity of Sturmian words coincides with their subword complexity, statement (iv) of Theorem~\ref{th:main} follows from this lemma via a straightforward induction.

\begin{proof}[Proof of Lemma~\ref{lemma:stability}]
Let $w_0\in\A^\N$ and $w_1\in\B^\N$ be two infinite words defined over disjoint finite alphabets. Let $a\in\A$ and $w:=\colour(w_0,a,w_1)$ be the coloring of the letter $a$ in $w_0$ by $w_1$. We assume that $b_{w_0}^k=p_{w_0}$ and $b_{w_1}^k=p_{w_1}$ for some $k\in\N_{>0}$ and aim to prove that $b_w^k=p_w$. To this end, consider two factors $u,v\in\lan(w)$ such that $u\sim_k v$. Our goal is to show that $u=v$.

We denote by $\sigma$ the substitution that maps $w$ back to $w_0$:
\[
    \begin{array}{rccl}
        \sigma:&\A\sqcup\B\setminus\{a\}&\longrightarrow&\A\\
        & i & \longmapsto & \begin{cases} i &\text{if $i\in\A$,} \\ a &\text{if $i\in\B$.}\end{cases}
    \end{array}
\]
Clearly, $\sigma(w)=w_0$, $\sigma(u),\sigma(v)\in\lan(w_0)$, $\pi_\B(w)=w_1$ and $\pi_\B(u),\pi_\B(v)\in\lan(w_1)$. Moreover,
\[
    u=\colour(\sigma(u),a,\pi_{\B}(u)) \hspace{0.5cm} \text{ and }\hspace{0.5cm} v=\colour(\sigma(v),a,\pi_\B(v)).
\]
Therefore, to prove that $u=v$ it is suficient to prove that $\pi_{\B}(u)=\pi_\B(v)$ and $\sigma(u)=\sigma(v)$ (it is even equivalent). Thanks to a reasoning similar to that used in the proof of Lemma~\ref{lemma:projections_kbin}, we have $\pi_\B(u)\sim_k\pi_\B(v)$ (see equation \eqref{eq:projections} in Section~\ref{sec:binary_projections}), and then $\pi_\B(u)=\pi_\B(v)$. We now prove that $\sigma(u)=\sigma(v)$. To begin with, we claim that for any $x\in\A^*$
\begin{equation}\label{eq:coloring}
    \binom{\sigma(u)}{x}=\sum\limits_{\substack{y\in(\A\sqcup\B\setminus\{a\})^{|x|} \\ \sigma(y)=x}} \binom{u}{y}.
\end{equation}
We temporarily postpone the proof of this claim. Since $u\sim_k v$, we have $\binom{u}{y}=\binom{v}{y}$ for every $y\in(\A\sqcup\B\setminus\{a\})^{\leq k}$. Combined with \eqref{eq:coloring}, this yields $\binom{\sigma(u)}{x}=\binom{\sigma(v)}{x}$ for every $x\in\A^{\leq k}$, \emph{i.e.}, $\sigma(u)\sim_k\sigma(v)$. Finally, since both $\sigma(u)$ and $\sigma(v)$ belong to $\lan(w_0)$, and since the $k$-binomial complexity of $w_0$ coincides with its subword complexity, it follows that $\sigma(u)=\sigma(v)$.

\medskip

To conclude, it only remains to prove relation \eqref{eq:coloring}. Writing $u=u_1u_2\ldots u_p$, where each $u_i$ is a letter in $\A\sqcup\B\setminus\{a\}$, we have $\sigma(u)=\sigma(u_1)\sigma(u_2)\ldots\sigma(u_p)$ where each $\sigma(u_i)\in\A$ is also letter. Therefore, by definition of binomial coefficients of words, we have
\[
    \begin{array}{c}
        \ds\binom{\sigma(u)}{x}:=\#\big\{(i_1,i_2,\ldots,i_{|x|}) \;\big |\; 1\leq i_1<i_2<\ldots<i_{|x|}\leq p \text{ and } \sigma(u_{i_1})\sigma(u_{i_2})\ldots\sigma(u_{i_{|x|}})=x\big\},\\
        \\
        \ds\binom{u}{y}:=\#\big\{(j_1,j_2,\ldots,j_{|y|}) \;\big |\; 1\leq j_1<j_2<\ldots<j_{|y|}\leq p \text{ and } u_{j_1}u_{j_2}\ldots u_{j_{|y|}}=y\big\}.
    \end{array}
\]
Hence, the identity \eqref{eq:coloring} is an immediate consequence of the set equality
\begin{equation}\label{eq:coloringb}
    E(u,x)=\bigsqcup\limits_{\substack{y\in(\A\sqcup\B\setminus\{a\})^{|x|} \\ \sigma(y)=x}} F(u,y)
\end{equation}
where
\[
    \begin{array}{l}
        E(u,x):=\big\{(i_1,i_2,\ldots,i_{|x|}) \;\big |\; 1\leq i_1<i_2<\ldots<i_{|x|}\leq p \text{ and } \sigma(u_{i_1})\sigma(u_{i_2})\ldots\sigma(u_{i_{|x|}})=x\big\},\\
        F(u,y):=\big\{(j_1,j_2,\ldots,j_{|y|}) \;\big |\; 1\leq j_1<j_2<\ldots<j_{|y|}\leq p \text{ and } u_{j_1}u_{j_2}\ldots u_{j_{|y|}}=y\big\}.
    \end{array}
\]
Let us prove that \eqref{eq:coloringb} holds. First, it is clear that the sets $(F(u,y))_y$ are disjoint. Indeed, if there exist $y,z\in (\A\sqcup\B\setminus\{a\})^{|x|}$ such that $F(u,y)\cap F(u,z)\neq\emptyset$, then there exists $(j_1,\ldots,j_{|x|})$ such that $y=u_{j_1}\ldots u_{j_{|x|}}=z$. Secondly, if $(i_1,\ldots,i_{|x|})\in E(u,x)$, then $\sigma(u_{i_1})\ldots\sigma(u_{i_{|x|}})=x$. Consequently, $y:=u_{i_1}\ldots u_{i_{|x|}}$ belongs to $(\A\sqcup\B\setminus\{a\})^{|x|}$ and satisfies $\sigma(y)=x$. In other words
\[
    (i_1,\ldots,i_{|x|})\in F(u,y)\subseteq \bigsqcup\limits_{\substack{z\in(\A\sqcup\B\setminus\{a\})^{|x|} \\ \sigma(z)=x}} F(u,z).
\]
Finally, if $(j_1,\ldots,j_{|x|})\in F(u,y)$ for some $y\in(\A\sqcup\B\setminus\{a\})^{|x|}$ satisfying $\sigma(y)=x$, then
\[
    x=\sigma(y)=\sigma(u_{j_1}\ldots u_{j_{|x|}})=\sigma(u_{j_1})\ldots\sigma(u_{j_{|x|}}),
\]
which means that $(j_1,\ldots,j_{|x|})$ belongs to $E(u,x)$. This completes the proof.
\end{proof}

%%%%%%%%%%%%%%%%%%%%%%%%%%%%%%%%%%%%%%%%%%%%%%%%%%%%%
\section{Open questions}\label{sec:conclusion}

We established a sufficient condition and a stability result for words whose $2$-binomial complexity is equal to their subword complexity. These two results allowed us to show that several families of words, such as $d$-ary $1$-balanced words, $d$-ary words with subword complexity $n\in\N_{>0}\mapsto n+(d-1)$, hypercubic billiard words, and colorings of Sturmian words by Sturmian words, all share this property. All these classes of words can be thought of as generalizations of Sturmian words.

However, we have also shown that the combinatorial structure of the Tribonacci word does not fall within this framework, and a fully non-computer-assisted proof that its $2$-binomial complexity coincides with its subword complexity is still missing. In this direction, even a non-computer-assisted proof that the $k$-binomial complexity of the Tribonacci word coincides with its subword complexity for some $k\geq 3$ would already be of interest.

We recall that, more generally, Lejeune, Rigo and Rosenfeld conjectured that the $2$-binomial complexity of any Arnoux--Rauzy word is equal to its subword complexity. Some modest numerical experiments carried out by the author of this paper agreed with this conjecture.

\begin{proposition}
Up to length $n=99$, each factor of an Arnoux--Rauzy word generated by a periodic directive sequence with period at most $5$ is alone in its $2$-binomial class (for a definition of directive sequence, see for instance the proposition in \cite[Section~2]{AR91}).
\end{proposition}

Since Arnoux--Rauzy words are themselves a generalization of Sturmian words, this raises the following open questions.

\begin{question}
Is there a natural class of words that can be seen as a \emph{genuine generalization of Sturmian words}, for which the $2$-binomial complexity of at least one (resp. all) of its elements differs from its subword complexity?
\end{question}

For example, Cassaigne--Selmer words (also known as $C$-adic words, \cite{CLL17,CLL22}) are also a fair arithmetic generalization of Sturmian words, and some modest numerical experiments carried out by the author of this paper suggest that their $2$-binomial complexity also coincides with their subword complexity.

\begin{proposition}
Up to length $n=99$, each factor of a Cassaigne--Selmer word generated by a periodic directive sequence with period at most $5$ is alone in its $2$-binomial class.
\end{proposition}

Since both Arnoux--Rauzy words and Cassaigne--Selmer words belong to the broader class of dendric words \cite{BDDLPRR15,BDDLP18}, it would also be interesting to investigate whether there exists a subclass of dendric words for which none of its elements satisfies this property.

\begin{question}
In the opposite direction, is there a class of words, \emph{genuinely unrelated to Sturmian words}, such that the $2$-binomial complexity, but not the $1$-binomial complexity, of at least one (resp. all) of its elements coincides with its subword complexity?
\end{question}

\medskip

Our last comment concerns the binary case. Proposition~\ref{prop:binary} states that the $2$-binomial complexity of any binary $1$-balanced word coincides with its subword complexity. Moreover, we saw that the converse does not hold. Indeed, words of the form $1^m 2^\omega$ (with $m\geq 2$) are $m$-balanced but not $(m-1)$-balanced and their $2$-binomial complexity coincides with their subword complexity. However, the $1$-binomial complexity of these words also coincides with their subword complexity, and it would be interesting to determine whether other counterexamples exist.

\begin{question}
Does the converse of Proposition~\ref{prop:binary} hold when restricted to words for which their $1$-binomial complexity does not coincide with their subword complexity?
\end{question}

%%%%%%%%%%%%%%%%%%%%%%%%%%%%%%%%%%%%%%%%%%%
% 
%    Acknowledgment and biblio  
%
%%%%%%%%%%%%%%%%%%%%%%%%%%%%%%%%%%%%%%%%%

\paragraph{Acknowledgments} The author would like to thank M\'elodie Andrieu for having introduced him to Combinatorics on Words and for her enthusiasm for this work.

\small
\bibliography{biblio_kbin}
\bibliographystyle{alpha}

\end{document}